\documentclass[11pt]{article}
\usepackage[margin=1in]{geometry}
\usepackage[T1]{fontenc}
\usepackage{lmodern}
\usepackage{times}
\usepackage{amsmath,amssymb,amsthm,mathtools}
\usepackage[colorlinks=true,linkcolor=blue,citecolor=blue,urlcolor=blue]{hyperref}
\usepackage{booktabs}

\newtheorem{theorem}{Theorem}[section]
\newtheorem{proposition}[theorem]
{Proposition}
\newtheorem{corollary}[theorem]{Corollary}
\newtheorem{lemma}[theorem]{Lemma}
\newtheorem{remark}[theorem]{Remark}
\newcommand{\cC}{\mathcal C}
\newcommand{\cB}{\mathcal B}
\newcommand{\cD}{\mathcal D}
\newcommand{\I}{\mathrm{I}}
\newcommand{\cM}{\mathcal M}
\newcommand{\Q}{\mathbb Q}
\newcommand{\Z}{\mathbb Z}
\newcommand{\C}{\mathbb C}
\renewcommand{\Vec}{\operatorname{Vec}}
\newcommand{\ad}{\mathrm{ad}}
\newcommand{\pt}{\mathrm{pt}}
\newcommand{\FPdim}{\mathrm{FPdim}}
\newcommand{\Gr}{\mathrm{Gr}}
\newcommand{\hxs}{\mathfrak{s}}
\newcommand{\hxt}{\mathfrak{t}}
\newcommand{\K}{\mathbb K}
\newcommand{\SL}{\operatorname{SL}}
\newcommand{\Ind}{\operatorname{Ind}}
\newcommand{\cO}{\mathcal O}
\newcommand{\sVec}{\mathrm{sVec}}
\newcommand{\one}{\mathbf 1}
\DeclareMathOperator{\Tr}{Tr}
\DeclareMathOperator{\Gal}{Gal}
\DeclareMathOperator{\FSexp}{FSexp}
\DeclareMathOperator{\rank}{rank}

\DeclareMathOperator{\ord}{ord}
\numberwithin{equation}{section}

\begin{document}
\title{Pre-modular fusion categories of global dimension $2p$}
\author{Zhiqiang Yu}
\date{}

\maketitle

\begin{abstract}
Let  $p>5$  be a  prime. In this paper, we show that a  pre-modular fusion category of
  global dimension $2p$ is either pointed or  tensor equivalent to $\text{Rep}(D_{2p})$.
\end{abstract}

\section{Introduction}
Throughout this paper, let $\C$ be the field of complex numbers, $\Q$ and $\bar{\Q}$ denote
the fields of rationals and its algebraic closure, respectively. For $n\in \mathbb{N}$, let $\Z_r:= \Z/r$.

A fusion category $\cC$ over the complex field $\C$ is a semisimple finite tensor category.
Braided fusion categories form an important class of fusion categories, with close connections to vertex operator algebras and conformal field theory \cite{DLN}, quantum groups at root of unity \cite{EGNO}, etc. Their structure and classification have been studied extensively \cite{BGNPRW,BNRW,DrGNO1,NRW,Schopieray2021,Yu2020}. Notice  that the M\"{u}ger center of a braided fusion category is symmetric, which then contains a maximal Tannakian fusion subcategory \cite{DrGNO1,EGNO}. De-equivariantization by this maximal Tannakian subcategory \cite{DrGNO1} reduces many structural results and properties  about braided fusion categories to two fundamental cases: non-degenerate and slightly degenerate fusion categories, whose M\"{u}ger centers are  $\Vec$ and $\sVec$, respectively [6].   If we further assume that they are spherical, then we get super-modular and modular fusion categories, see Section \ref{preliminaries} for definition.

The   $S$-matrix and $T$-matrix (called modular data, see Section \ref{preliminaries}) of modular and super-modular fusion categories encode structural information and satisfy strong algebraic and arithmetic constraints \cite{BNRW,DLN,NRW}. In the modular case, a suitably normalized modular data give rise to congruence representations of the modular group $\SL_2(\Z)$ \cite{DLN,NRW}; in the super-modular case, the reduced $S$-matrix and the matrix determined by  the squared twists of simple objects instead give rise to a projective congruence representation of the  subgroup
$\Gamma_\theta$ of $\SL_2(\Z)$ \cite{BRWZ}.
These connections provide tools from algebraic number theory and representation theory for studying the structure of such fusion categories. 

One of the main topics  in the study of (super-) 
modular fusion categories is their classification. 
Currently, given a specific global dimension, the structure of (super-)modular fusion categories $\cC$ are far away being well-studied. 
Let $p$ be a prime. Pre-modular fusion categories $\cC$ of global dimension   $p$ are pointed if $p\neq 5$ \cite[Proposition 5.11]{Schopieray2021},  that is, every simple object  of $\cC$ is invertible, hence $\cC$ is determined by a pre-metric group $(\Z_p,\eta)$ with quadratic form $\eta$ \cite{DrGNO1}. Non-simple (i.e., there is no non-trivial fusion subcategory) modular fusion categories of global  dimension $p^2$ are completely classified in \cite[Theorem 4.14]{Yu2023}.   

In this paper, we combine algebraic and arithmetic methods, especially congruence representations of $\SL_2(\Z)$, to study pre-modular fusion categories of global dimension $2p$, where $p>5$ is prime. For $p\leq 5$,  their structures are known by \cite{Yu2022}.   Our main result is stated as follows. 
\begin{theorem}
Let $\cC$ be a pre-modular fusion category of global dimension $2p$, where $p>5$ is a prime. Assume $\cC'\subseteq\sVec$, then $\cC$ is pointed.   
\end{theorem}

The paper is organized as follows. In Section \ref{preliminaries}, we recall some basic notions and notations of   (pre-modular) fusion categories, such as  global dimensions, formal codegrees, modular data,  etc. In Section \ref{sectionsuper}, we show that a super-modular fusion category of global dimension $2p$ is pointed (see Theorem \ref{supermodularpointed}), where $p>5$ is a prime. In Section \ref{sectionmodular}, for a modular fusion category $\cC$ of global dimension $2p$, we show $\cC$ is also pointed in Theorem \ref{modularpointed}.
\section{Preliminaries}\label{preliminaries}
\subsection{Spherical fusion categories}
Let $\cC$ be a fusion category,   $\cO(\cC)$  the set of isomorphism classes of simple objects of  $\cC$, the cardinal of $\cO(\cC)$ is called the rank (denote by $\rank(\cC)$) of $\cC$. Then the Frobenius-Perron homomorphism $\FPdim(-):\Gr(\cC)\to\C$ is the unique ring homomorphism such that $\FPdim(X)\geq1$ for all  $X$, $\FPdim(X)$ is an algebraic integer and is defined as the Frobenius-Perron dimension of the object $X$. The sum
\begin{align*}
\FPdim(\cC):=\sum_{X\in\Q(\C)}\FPdim(X)^2
\end{align*}
is called  the Frobenius-Perron dimension of fusion category $\cC$.

The  subcategory tensor generated by $X\otimes X^*$ for all $X\in\cO(\cC)$ is called the adjoint fusion subcategory $\cC_\ad$ of  $\cC$. 
A simple object of $\cC$ is called invertible if $X\otimes X^*=\one$, the unit object, equivalently, $\FPdim(X)=1$. A fusion category $\cC$ is pointed if $\cO(\cC)$ is a finite group, where the group multiplication is   the tensor product $\otimes$. In the following, we use $\cC_\pt$ to denote the maximal pointed subcategory of $\cC$, i.e., the fusion subcategory generated by invertible simple objects of $\cC$.  

Let $\cC$ be a spherical fusion category with a spherical structure $j$, that is, $j$ is a natural isomorphism from identity tensor functor $\text{id}_\cC$ to the double dual tensor functor $(-)^{**}$ such that $\dim_j(X)=\dim_j(X^*)$, $\forall X\in\cC$, where   $\dim_j(X)$ is the the dimension of an object $X$ is defined by the (categorical) trace of $\text{id}_X$:
\begin{align*}
\dim_j(X)=\text{Tr}(\text{id}_X):=\text{ev}_X\circ(j_X\circ \text{id}_X)\otimes \text{id}_{X^*}\circ\text{coev}_X.
\end{align*}
Then we define the global (or categorical) dimension of the spherical fusion category $\cC$ as
\begin{align*}
\dim(\cC):=\sum_{X\in\cO(\cC)}\dim_j(X)^2.
\end{align*}
Then $\dim_j(-)$ induces a homomorphism from $\Gr(\cC)$ to $\cC$
In this paper, we will fix a spherical structure $j$ and denote $\dim_j(-)$ by $\dim(-)$. For more properties of global dimension, see \cite{EGNO,ENO,Ostrik}. 
 
Notice that the complexified ring $\Gr(\cC)\otimes_\Z\cC$ of a fusion category $\cC$ is semisimple. Given an irreducible representation $\chi$ of $\Gr(\cC)\otimes_\Z\cC$,  the element $
\alpha_\chi=\sum_{X\in\cO(\cC)}\text{tr}_\chi(X)X^*$
is central, where $\text{tr}_\chi$ is the ordinary trace function on the representation $\chi$. Moreover, $\chi'(\alpha_\chi)=0$ if $\chi\ncong\chi'$ and $f_\chi:=\chi(\alpha_\chi)$ is a positive algebraic integer \cite{Lusztig}, $f_\chi$ is called a formal codegree of fusion category $\cC$ \cite{Ostrik2009}. For example, $\FPdim(\cC)$ and $\dim(\cC)$ are formal codegrees of $\cC$. It follows from \cite[Corollary 2.14]{Ostrik2015} that $\frac{\dim(\cC)}{f_\chi}$ is always a totally positive algebraic integer.

\begin{lemma}\label{lemmaFPdimension}
Let $\cC$ be a fusion category such that $\cC_{\ad}=\cC$, then the formal codegree  $\FPdim(\cC)$ has multiplicity one.
\end{lemma}
\begin{proof}
Let $R=\sum_{X\in\cO(\cC)}XX^*$, its fusion matrix
$N_R=\sum_X N_XN_{X^*}$ is a positive definite matrix. The summands
$XX^*$ generate $\cC_{\ad}=\cC$, so $N_R$ is irreducible, and the  eigenvalues of $N_R$ are $\dim(\chi)\,f_\chi$, which is less than or equal to $\FPdim(\cC)$.  The Frobenius-Perron theorem
makes the  eigenvalue $\FPdim(R)=\FPdim(\cC)$ simple.
\end{proof}

\subsection{Modular and super-modular fusion categories}\label{subsectionmodsuper}
Let $\cC$ be a braided fusion category with braiding $c$ and $\cD$  a fusion subcategory of $\cC$. Two objects $X,Y$ are said to be centralize each other if
$c_{Y,X} c_{X,Y}=\text{id}_{X\otimes Y}.$
Then the centralizer $\cD_\cC'$ of $\cD$ in $\cC$ is the full subcategory of objects that centralize every object of $\cD$. We call $\cC':=\cC_\cC'$ the M\"{u}ger center of $\cC$ \cite{Mu1}.

Let $\cC$ be a braided spherical (i.e., pre-modular) fusion category with ribbon structure $\theta$. The matrices $S=(S_{X,Y^*})$ and $T=(\delta_{X,Y}\theta_X)$ are called the $S$-matrix and $T$-matrix of $\cC$, respectively, where $S_{X,Y}=\text{Tr}(c_{Y,X}c_{X,Y})$, $X,Y\in\cO(\cC)$.
A modular fusion category $\cC$ is a pre-modular fusion category such that the $S$-matrix $S$ is non-degenerate, equivalently $\cC'=\Vec$ \cite{DrGNO1,EGNO,Mu1}. Moreover, the  $S$-matrix of a modular fusion category $\cC$ determines the multiplication of the Grothendieck ring $\Gr(\cC)$ by the famous Verlinde formula \cite{EGNO} which states that for any $X,Y,Z\in\cO(\cC)$,
\begin{align}\label{Verlinde}
N_{X,Y}^Z:=\dim_\C(\text{Hom}(X\otimes Y,Z))=\frac{1}{\dim(\cC)}\sum_{W\in\cO(\cC)}\frac{S_{X,W}S_{Y,W}S_{Z^*,W}}{\dim(W)}.
\end{align}

Let $X$ be a simple object of a (pre-)modular fusion category $\cC$,  then $h_X(Y):=\frac{S_{X,Y}}{\dim(X)}$ defines a ring homomorphism from $\Gr(\cC)$ to $\C$, $\forall Y\in\cO(\cC)$. And the Verlinde formula (\ref{Verlinde})  also implies that  $\{h_X(-)|X\in\cO(\cC)\}$ is  the set of ring homomorphisms of modular fusion category $\cC$. And notice that the formal codegree determined by homomorphism $h_X(-)$ is $\frac{\dim(\cC)}{\dim(X)^2}$. Since $\sigma\circ h_X(-)$ is also a ring homomorphism  of $\Gr(\cC)$, where $\sigma\in\Gal(\overline\Q/\Q)$,  there is a unique simple object $Y$ such that $\sigma\circ h_X(-)=h_Y(-)$. Hence, there is a unique  permutation $\hat\sigma$ of $\cO(\cC)$ such that $\hat\sigma(X)=Y$ and $\sigma\circ h_X(-)=h_{\hat\sigma(X)}(-)$. We call the subset
$\cO_X:=\{Y\in\cO(\cC)|Y=\hat\sigma(X) ~\text{for some $\sigma\in\Gal(\overline\Q/\Q)$}\}$
 the Galois orbit of  $X$.

The modular data $(S,T)$ of a modular fusion category $\cC$ is also connected closely with the congruence representations of the modular group $\SL_2(\Z)$ \cite{BNRW,DLN}, which  is generated by $\hxs=\left(
                                                      \begin{array}{cc}
                                                        0 & -1 \\
                                                        1 & 0 \\
                                                      \end{array}
                                                    \right)
$ and $\hxt=\left(
              \begin{array}{cc}
                1 & 1 \\
                0 & 1 \\
              \end{array}
            \right)
$ with relations $\hxs^4=1$ and $(\hxs\hxt)^3=\hxs^2$. Specifically, the morphism $\hxs\mapsto\frac{1}{\sqrt{\dim(\cC)}}S$, $\hxt\mapsto T$ defines a projective representation of $\SL_2(\Z)$ \cite[Theorem 8.16.1]{EGNO}, where $\sqrt{\dim(\cC)}$ is the positive square root of $\dim(\cC)$.
Let $\tau^\pm(\cC):=\sum_{X\in\cO(\cC)}\dim(X)^2\theta_X^\pm$
be the Gauss sums of $\cC$ \cite{DrGNO1,EGNO}, then $\xi(\cC):=\frac{\tau^+(\cC)}{\sqrt{\dim(\cC)}}$ is called the multiplicative central charge of $\cC$. It follows from \cite{BNRW,DLN} that there always exists a third root $\xi$ of  $\xi(\cC)$ such that  $s:=\rho_\cC(\hxs)=\frac{1}{\sqrt{\dim(\cC)}}S$ and $t:=\rho_\cC(\hxt)=\frac{1}{\xi}T$ defines a finite-dimensional congruence representation $\rho_\cC$ of $\SL_2(\Z)$ of level $N$,  that is, $\ker(\rho_\cC)$ contains a congruence subgroup $\Gamma(N)$ of level $N=\text{ord}(t)$. Moreover, $\Q(S)\subseteq\Q(T)\subseteq\Q(t)$
\cite[Theorem II]{DLN}. Let $t_X:=\theta_X/\xi$, the normalized ribbon structure of the simple object $X$,  $\forall X\in\cO(\cC)$, then we have the Galois symmetry \cite[Theorem II]{DLN}, that is, $\sigma^2(t_X)=t_{\hat\sigma(X)}$, $\forall \sigma\in\Gal(\Q(t)/\Q)$.

A pre-modular fusion category $\cC$ is super-modular if $\cC'=\sVec$.  Let $f$ be the nontrivial simple object of $\sVec$. Then tensor product of $f$ induces a partition on $\cO(\cC)=\mathrm{I}_0\cup\mathrm{I}_1$, $\mathrm{I}_1=\{Y|Y\cong f\otimes X, X\in\mathrm{I}_0\}$ and $\one\in \mathrm{I}_0$. The $S$-matrix of $\cC$ has the form $\left(\begin{array}{cc}
    \hat{S} & \varepsilon\hat{S} \\
    \varepsilon\hat{S} & \hat{S}
\end{array}\right)$ with $\hat{S}$ being a non-degenerate symmetric   matrix, where $\varepsilon:=d_f\in\{\pm1\}$. Moreover, 
$\hat S^2=\frac{\dim(\cC)}{2}C,  (C)_{XY}=\delta_{X,Y^*},
$
and the   characters
 $h_X(Y):=\frac{\hat S_{X_i,X_j}}{d_{X_j}}$
are also permuted by Galois automorphisms, namely, $\sigma\circ h_X(-)=h_{\hat{\sigma}(X)}(-)$, see \cite{BGNPRW} for details. And it further proved in \cite[Corollary 3.4]{Yu2020} that $\frac{\dim(\cC)}{2d_X^2}$ is an algebraic integer, $\forall X\in\mathrm{I}_0$.

However, $(\hat{S},\hat{T})$ does  not admits congruence representation of $\SL_2(\Z)$ in general, where $\hat{T}=(\delta_{X,Y}\theta_X)$. Indeed, if $\cC$ admits a minimal modular extension (see \cite{Mu1}), then it was proved in \cite[Theorem 3.1]{BRWZ} that there exists a projective congruence representation $\hat{\rho}:\Gamma_\theta\to \SL_r(\C)$of level $N=\ord(T)$, where \begin{align*}
 \Gamma_\theta:=\langle \hxs,\hxt^2\rangle= \bigg\{\left(\begin{array}{cc}
    a & b \\
    c & d
\end{array}\right)\in\SL_2(\Z)\bigg| ac\equiv bd\equiv 0\pmod 2\bigg\},\quad\hat{\rho}:\hxs\mapsto\hat{S},\hxt^2\mapsto\hat{T}^2.
\end{align*}

In the following, for a fusion category $\cC$, we always use the following notations:  
$\Q(\zeta_N)^+=\Q(\zeta_N+\zeta_N^{-1})$, 
$\K=\Q(d_X:X\in\cO(\cC)), h=[\K:\Q]$, $V_\cC:=\C^{\cO(\cC)}$ with its standard inner product and basis 
 $\{e_X|X\in\cO(\cC)\}$. For any Galois extension $L/\Q$, we denote the norm and trace of an algebraic integer $\alpha\in L$ by $N_{L/\Q}(\alpha)=\Pi_{\sigma\in\Gal{L/\Q}}\sigma(\alpha)$ and $\Tr_{L/\Q}(\alpha)=\sum_{\sigma\in\Gal{L/\Q}}\sigma(\alpha)$, respectively.

\section{Super-modular fusion category of global dimension $2p$}\label{sectionsuper}

The following lemma is easy, we include it for later use.
\begin{lemma}\label{lemmadimension}
     Let $\cC$ be a pre-modular fusion category of global dimension being a square-free integer $d$. If its  M\"{u}ger center $\cC'\subseteq\sVec$ and $\dim(X)\in\Z$ for all $X\in\cO(\cC)$, then $\cC$ is pointed.
\end{lemma}
\begin{proof}
    Notice that we have $d_X^2$ divide $d$ for all $X\in\cO(\cC)$ by \cite[Proposition 8.14.6]{EGNO} and \cite[Corollary 3.4]{Yu2020}, which implies $d_X=\pm 1$. Therefore, $\rank(\cC)=d$ and $\cC$ is pointed by \cite[Remark 4.2.3]{Ostrik}.
\end{proof}

\begin{lemma}
    Let $\cC$ be a pre-modular fusion category of global dimension $2p$, where $p>5$ is a  prime. If its  M\"{u}ger center $\cC'$ is non-trivial, then $\cC'=\sVec$, $\cC$ is pointed,  or $\cC\cong\text{Rep}(D_{2p})$ as tensor category.
\end{lemma}
\begin{proof}
Notice that $\cC'$ is super-Tannakian and $\dim(\cC')\mid \dim(\cC)$ by \cite[Theorem 3.1]{Yu2022}, so $\dim(\cC')=2,p,2p$. If $\dim(\cC')=2p$, then $\cC$ is symmetric, so $\cC$ is either pointed or tensor equivalent to $\text{Rep}(D_{2p})$. Assume $\dim(\cC')$ is a prime, if $\cC'$ is Tannakian, then the de-equivariantization by the Tannakian subcategory is a pointed fusion category of global dimension being a prime \cite[Proposition 5.11]{Schopieray2021}, so $\cC\cong \cC(\Z_r,\eta)^{\Z_s}$, where these two distinct primes $r,s\in\{2,p\}$, which is pointed or equivalent to $\text{Rep}(D_{2p})$. 
\end{proof}
So, in order to give a complete classification of pre-modular fusion categories of global dimension $2p$, we only need to solve the cases when $\cC'\subseteq\sVec$, i.e. $\cC$ is modular or $\cC'$ is super-modular.  Notice that pre-modular fusion categories  of global dimension $10$ are  classified in \cite{Yu2022}, so we assume $p>5$ below. 

We first classify super-modular fusion category $\cC$ such that $\frac{\dim(\cC)}{2}$ has square-free norm. Let $\cC(\mathfrak{sl}_2,3)_\ad$ be the Fibonacci modular fusion category of dimension $\frac{5+\sqrt{5}}{2}$, which has rank $2$, see \cite{EGNO}.
\begin{proposition}\label{propositionsquarefree}
    Let $\cC$ be a  super-modular fusion category, and let $d$ be a positive square-free integer. If $\dim(\cC)\not\in\Z$ and $N_{\Q(\dim(\cC))/\Q}(\dim(\cC)/2)=d$. Then $\K=\Q(\dim(\cC))=\Q(\sqrt{5})$,  $\rank(\cC)=4$ and $d=5$. Moreover, $\cC\cong\sVec\boxtimes \cD$, where $\cD$ is a Galois conjugate of $\cC(\mathfrak{sl}_2,3)_\ad$.
\end{proposition} 
\begin{proof}
 Notice that  $\frac{\dim(\cC)}{2}$ is   also a $d$-number by \cite[Lemma 2.7]{Ostrik2009}, since any $\sigma\in\Gal(\overline{\Q}/\Q)$,
\begin{align*}
\sigma\left(\frac{\dim(\cC)}{2}\right)\bigg/\frac{\dim(\cC)}{2}=\sigma(\dim(\cC))/\dim(\cC)
\end{align*}
is an algebraic unit. Thus, for any integer $q\in\Z_{\geq2}$, \cite[Corollary 4.3]{Schopieray2021} says that $q$ does not divide $\frac{\dim(\cC)}{2}$, and the argument of \cite[Theorem 4.4]{Schopieray2021} shows
 \begin{align*}
\sum_{\sigma\in\Gal(\Q(\dim(\cC))/\Q)}  \frac{1}{\sigma(\dim(\cC)/2)}
=\sum_{\sigma\in\Gal(\Q(\dim(\cC))/\Q)} \frac{2}{\sigma(\dim(\cC))}\geq 1.
\end{align*}
 Meanwhile $\frac{\dim(\cC)}{2}=\sum_{X\in\I_0}\dim(X)^2$, then   we obtain
 \begin{align*}
 \frac{1}{2}\leq \sum_{\sigma\in\Gal(\Q(\dim(\cC))/\Q)}\frac{1}{\sigma(\dim(\cC))} =    \sum_{X\in\cO_\one} \frac{1}{\frac{\dim(\cC)}{\dim(X)^2}}\leq    \sum_{X\in\I_0} \frac{1}{\frac{\dim(\cC)}{\dim(X)^2}}=\frac{1}{2},
\end{align*}
so $\I_0=\cO_\one$  and $\rank(\cC)=2[\K:\Q]$. 

Let $ A=\bigl(\sigma(d_X)\bigr)_{
       \sigma\in\Gal(\K/\Q),\,X\in \I_0}$.
For any $\sigma,\tau\in\Gal(\K/\Q)$, the orthogonality of $\hat{S}$ and Galois symmetry for $\cC$ in Subsection \ref{subsectionmodsuper} imply
$ \sum_{X\in \I_0}\sigma(d_X)\tau(d_X)
=\delta_{\sigma,\tau}\sigma\left(\frac{\dim(\cC)}{2}\right).$ Therefore
\begin{equation}\label{eq:Gram}
AA^{T}=\text{diag}\left(\sigma\left(\frac{\dim(\cC)}{2}\right)\right)_\sigma,
 \quad (\det A)^2=N_{\K/\Q}\left(\frac{\dim(\cC)}{2}\right)=d.
\end{equation}
 In particular,    $A$ is invertible, then the $\{d_X|X\in\I_0\}$ is linearly independent over $\Q$, which forms a $\Q$-basis of $\K$. Let $\Lambda=\sum_{X\in I}\Z d_X\subseteq\mathcal O_K$, then $\Lambda$ is a full lattice of algebraic integers, and its discriminant is
$ \text{disc}(\Lambda)=(\det A)^2=d$.
The discriminant--index formula \cite[Corollay 1.4.3]{Cohen} yields
$ d=\text{disc}(\K) [\mathcal O_\K:\Lambda]^2$, where $[\cO_\K:\Lambda]$ is the index of additive group $\Lambda$ in $\cO_\K$,
so $[\mathcal O_\K:\Lambda]=1$ and
$\text{disc}(\K)=d$ as $d$ is square-free.
Notice that $\K/\Q$ is Galois, every entry of $A$ belongs to
$\K$. Equation~\eqref{eq:Gram} shows that $\det A=\pm\sqrt d\in \K$, consequently,
$\Q(\sqrt d)\subseteq \K$.
Because $d>1$ is square-free,
\[
 \text{disc}(\Q(\sqrt d))=
 \begin{cases}
d,&d\equiv1\pmod4,\\
 4d,&d\not\equiv1\pmod4.
 \end{cases}
\]
The discriminant formula for the tower $\K/\Q(\sqrt d)/\Q$ \cite[Theorem 2.5.1]{Cohen} gives \begin{align*}
 \text{disc}(\K)=\text{disc}(\Q(\sqrt d))^{[\K:\Q(\sqrt d)]}
           N_{\Q(\sqrt d)/\Q}(\mathfrak d_{\K/\Q(\sqrt d)})   ,
\end{align*}
where $\mathfrak d_{\K/\Q(\sqrt d)}$ is the relative discriminant ideal and its
norm is a positive integer; see \cite[Section 2.2.4]{Cohen}.
Since $\text{disc}(\K)=d$ and $\text{disc}(\Q(\sqrt d))\ge d$, we obtain $
 d\ge d^{[\K:\Q(\sqrt d)]}=d^{h/2}$, so $h=2$. It follows from  \cite[Theorem 3.1]{BGNPRW} that  $\cC\cong\sVec\boxtimes \cD$, where $\cD$ is a modular fusion category of rank $2$, which then equivalent to a   Galois conjugate of $\cC(\mathfrak{sl}_2,3)_\ad$.
\end{proof}

\begin{proposition}\label{propositionfield}
Let $\cC$ be a super-modular fusion category of global dimension $2p$. Let $m$ be the cardinal of $\{X\in \I_0\mid d_X^2=1\}$. If $\cC$ is not pointed, then
$
 2\le h\le m,~ hm\le p,
 ~ h<\sqrt p.
$
\end{proposition}

\begin{proof}
The dimension character is $h_\one(-)=\dim(-)$, its Galois orbit   has size $h$ (i.e., the cardinal of $\cO_\one$) by definition. As $\dim(\cC)$ is an integer,  \cite[Corollary 3.4]{Yu2020} shows $h\le m$. By Lemma \ref{lemmadimension},  $h\geq2$.

For $\sigma\in\Gal(\K/\Q)$ define
$
 v_\sigma=(\sigma(d_X))_{X\in \I_0}\in\C^{\I_0}$. The Galois symmetry shows that $v_\sigma$ is the row of $\widehat S$ indexed by
$\hat{\sigma}(\one)$, divided by $d_{\hat\sigma(\one)}$. Notice that $\hat{\sigma}(\one)$ is self-dual, thus \cite[Proposition 2.7]{BGNPRW}
\begin{equation}\label{equationorthogonality}
 v_\sigma\cdot v_\tau=p\delta_{\sigma,\tau}.
\end{equation}
 Meanwhile, let   $w:=\sum_\sigma v_\sigma$, for each $Y$ with $d_Y^2=1$,  the corresponding coordinate of  $w$
is $hd_Y=\pm h$,  we have 
$hp=\|w\|^2\geq\sum_{Y\in\I_0,d_Y^2=1}h^2=mh^2$.  Hence  $hm\leq p$ and $h<\sqrt p$.

Let $\mathfrak l$ be a prime ideal of $\mathcal O_\K$ such that $\mathfrak{l}\cap\Z=\langle \ell\rangle$, where $\cO_\K$ is the ring of algebraic integers of $\K$. Suppose that $1\ne\sigma\in\Gal(\K/\mathbb Q)$
fixes $\mathfrak l$ and induces the identity map on
$\mathcal O_\K/\mathfrak l$. Since $d_X\in\mathcal O_\K$, $\forall X\in\cO(\cC)$, we have
$
\sigma(d_X)\equiv d_X\pmod{\mathfrak l}$.
Applying Equation \eqref{equationorthogonality} with $\tau=1$ gives
\[
 0=\sum_{X\in\I_0}\sigma(d_X)d_X
   \equiv\sum_{X\in\I_0} d_X^2=p\pmod{\mathfrak l}.
\]
Thus $\ell=p$.  The abelian extension $\K/\Q$ is unramified at every
finite prime other than $p$, so its conductor is a power of $p$.
Hence $\K\subseteq\Q(\zeta_{p^a})$ for some $a\ge1$.
Since $h<\sqrt{p}$, the $p$-Sylow subgroup of
$\Gal(\Q(\zeta_{p^a})/\Q)$ acts trivially on $\K$.
It follows that $\K\subseteq\Q(\zeta_p)$. 
\end{proof}

\begin{corollary}\label{corollaryp=7}
    Super-modular fusion category of global dimension $14$ is pointed.
\end{corollary}
  \begin{proof}
     If $\cC$ is not pointed, then it follows from Proposition \ref{propositionfield} that $[K:\Q]<\sqrt{7}$ and $[\K:\Q]\mid 3$, it is a contradiction. 
  \end{proof} 

\begin{lemma}\label{lemmatrace}
Let $p\ge11$ be prime, 
$n=[\Q(\zeta_p)^+:\Q]=(p-1)/2$.  Suppose $L\subseteq \Q(\zeta_p)^+$ and has degree
$2\le k<\sqrt p$.  Then 
\begin{equation}\label{equationtracebound}
 \Tr_{\Q(\zeta_p)^+/\Q}(\alpha^2)
 \ge \frac{np(k-1)}{k^2}
 \ge p\quad \forall \alpha\in\cO_L\setminus\Q.
\end{equation}
\end{lemma}

\begin{proof}
Let  $H=\Gal(\Q(\zeta_p)/L)$.  Because $L$ is real, $-1\in H$.
Write $A_1,\ldots,A_k$ for the cosets of $H$ in
$(\Z/p\Z)^\times$, and define the Gaussian periods $
 \eta_j=\sum_{a\in A_j}\zeta_p^a$.
The elements $\zeta_p,\zeta_p^2,\ldots,\zeta_p^{p-1}$ form an
integral basis of $\Q(\zeta_p)$, permuted by $H$.
Taking $H$-invariants shows that $\eta_1,\ldots,\eta_k$ form an
integral basis of $L$.

For $a,b\in(\Z/p\Z)^\times$, it is well-known that 
\[
 \Tr_{\Q(\zeta_p)/\Q}(\zeta_p^{a+b})=
 \begin{cases}
 p-1,&a+b\equiv0\pmod p,\\
 -1,&a+b\not\equiv0\pmod p.
 \end{cases}
\]
There are $h$ pairs with $a+b=0$ in $A_i\times A_j$ when $i=j$,
and none when $i\ne j$.  Dividing the resulting trace by $h$ gives
$ \Tr_{L/\Q}(\eta_i\eta_j)=p\delta_{ij}-h$.
Write $\alpha=\sum_j a_j\eta_j$, with $a_j\in\Z$.
Then
\begin{align}
 \Tr_{L/\Q}(\alpha^2)
 =p\sum_j a_j^2-h\Bigl(\sum_j a_j\Bigr)^2\notag
 =\frac pk\sum_{i<j}(a_i-a_j)^2
   +\frac1k\Bigl(\sum_j a_j\Bigr)^2.
 \label{eq:trace-variance}
\end{align}
Since $\sum_j\eta_j=-1$, rationality of $\alpha$ is equivalent to
all $a_j$ being equal.
For nonconstant integer coefficients,
$
 \sum_{i<j}(a_i-a_j)^2\ge k-1.
$
Indeed, if the minimum coefficient occurs $r$ times, the $r(k-r)$
pairs crossing from these entries to the others each contribute at
least $1$, and $r(k-r)\ge k-1$.
Therefore
\[
 \Tr_{L/\Q}(\alpha^2)\ge\frac{p(k-1)}k,
\]
and multiplication by $[\Q(\zeta_p)^+:L]=n/k$ gives the first inequality in
\eqref{equationtracebound}.

For the second inequality, if $k\ge3$, then $k^2<p$ implies
$p-1\ge k^2$, hence  $
 n(k-1)=\frac{(p-1)(k-1)}2\ge k^2$;
If $k=2$, then $p\ge11$ gives $n\ge5\ge4=k^2$.
Thus $n(k-1)\ge k^2$ in both cases.
\end{proof}

\begin{remark}\label{rem:units}
The result in Lemma~\ref{lemmatrace} is not a lower bound for arbitrary totally positive
units.  For example, let
$ \alpha=\zeta_p+\zeta_p^{-1}$, it is well-known that $\alpha=\zeta_p+\zeta_p^{-1}$ generates $\Q(\zeta_p)^+$, and
$
 \Tr_{\Q(\zeta_p)^+/\Q}(\alpha^2)=p-2<p.
$
The field over which the trace is taken also matters.  For $p=13$,
take $L=\Q(\sqrt{13})$ and
$\alpha=\frac{3+\sqrt{13}}2,
 N_{L/\Q}(\alpha)=-1.
$
Then $\alpha^2$ is a totally positive unit, and
$
 \Tr_{L/\Q}(\alpha^2)=11<13, 
 \Tr_{\Q(\zeta_p)^+/\Q}(\alpha^2)=3\cdot11=33$.
\end{remark}

Let $\cC$ be a super-modular fusion category with  $\cC'=\sVec=\langle f\rangle$, $\cC^\text{rev}$ be the reverse braided fusion category of $\cC$, notice that the modular data of $\cC^\text{rev}$ are $(\bar S,T^{-1})$. Let 
$ \cM=\mathcal{Z}(\cC)$ be the Drinfeld center of $\cC$, then
\cite[Corollary~4.4]{DNO} shows that there is  a braided tensor equivalence $
\cB:=\sVec_{\mathcal{Z}(\cC)}'\cong\cC\boxtimes_{\sVec}\cC^{\mathrm{rev}}$,
where $\cC\boxtimes_{\sVec}\cC^{\mathrm{rev}}$ is the relative Deligne tensor product over $\sVec$, see \cite{DNO} for details. 
\begin{theorem}
Let $\cC$ be a super-modular fusion category. Let $N_0=\FSexp(\cC)$ and $N=4N_0$.
Then there is a congruence representation
$\phi:\Gamma_\theta\longrightarrow\SL(\C^{\I_0\times \I_0})$ 
such that
\begin{equation}\label{eq:center-representation}
 \phi(\hxs)=s\otimes\overline s,
 \quad \phi(\hxt^2)=\hat{T}_\cC^2\otimes \hat{T}_\cC^{-2},
 \quad \Gamma(N)\subseteq\ker (\phi), 
\end{equation}
where $s:=\frac{\sqrt{2}}{\sqrt{\dim(\cC)}}\hat{S}_\cC$. For $a\in(\Z_N)^\times$, let $\sigma_a(\zeta_N)=\zeta_N^a$,
and let $D_a$ denote the residue class of
$\text{diag}(a,a^{-1})$ in the image of $\Gamma_\theta$ modulo $N$, $G_{\sigma_a}:=\sigma_a(s)s^{-1}$.
Then we have 
\begin{equation}\label{equationdiagonalGaloistwist}
 R(D_a)=G_{\sigma_a}\otimes G_{\sigma_a},\quad 
\sigma_a^2(\theta_X^2/\theta_Y^2)=\theta_{\hat{\sigma}_a(X)}^2/\theta_{\hat{\sigma}_a(Y)}^2,
 \qquad
 \theta_{\hat{\sigma}(X)}^2=\theta_{\hat{\sigma}_a(\one)}^2\,\sigma_a^2(\theta_X^2).
\end{equation}
\end{theorem}

\begin{proof}
It follows from \cite[Corollary 4.4]{DNO} that the relative product $\cC\boxtimes_{\sVec}\cC^{\mathrm{rev}}$ 
is obtained using
the central algebra
$A=\one\boxtimes\one\oplus f\boxtimes f$, tensoring with $f\boxtimes f$ acts freely on simple pairs, the free $A$-modules are therefore simple, with pairs identified by
$(X,Y)\sim(f\otimes X,f\otimes Y)$.
After choosing representatives for the remaining action of $f$,
the reduced simple labels of $\cB$ are indexed by $J:=\I_0\times \I_0$, its normalized reduced
data are
\begin{equation}\label{eq:relative-data}
 s_{\cB}=s\otimes\overline s,
 \qquad \hat{T}_\cB=\hat{T}_\cC\otimes \hat{T}_\cC^{-1}.
\end{equation}

The Drinfeld center $\mathcal{Z}(\cC)$ has its canonical modular representation $\rho$ \cite{DLN,EGNO}, with
\[\rho(\hxs)=s_{\cM}:=\frac{S_{\cM}}{\dim(\cC)},
 \quad \rho(\hxt):=T_{\cM},
 \quad \Gamma(N_0)\subseteq\ker(\rho),
\]
where $(S_\cM,T_\cM)$ is the modular data of $\cM$. Its matrices have entries in $\Q(\zeta_{N_0})$. Let $
 W:=\operatorname{span}\{e_Z+\varepsilon e_{f\otimes Z}:Z\in J\}
 \subseteq V_\cM$.
Then $s_{\cM}$ and $T_{\cM}^2$ preserve $W$.
In fact,  
$ (S_{\cM})_{Y,f\otimes Z}
=\varepsilon\mu_Y(S_{\cM})_{Y,Z}$ by \cite[Section 3]{BRWZ}, where  $\mu_Y\in\{1,-1\}$.
Thus the coefficient of $e_Y$ in $S_\cM(e_Z+\varepsilon e_{f\otimes Z})$ is $2(S_\cM)_{Y,Z}$ if $Y\in\cB$, otherwise it is $0$. Hence, for any $Z\in\cO(\cM)$, we obtain 
\begin{align*}
 s_\cM (e_Z+\varepsilon e_{f\otimes Z}) =\frac{1}{\dim(\cC)}\sum_{Y\in\cO(\cM)}[(S_\cM)_{Y,Z}+\varepsilon (S_\cM)_{Y,f\otimes Z}]e_Y=\frac{2}{\dim(\cC)}\sum_{Y\in\cO(\cB)}(S_\cM)_{Y,Z} e_Y,
\end{align*}
in particular, for $Z\in J$, $ s_\cM (e_Z+\varepsilon e_{f\otimes Z})=\frac{2}{\dim(\cC)}\sum_{Y\in J}(\hat{S}_\cB)_{Y,Z} e_Y$.
Also we have $\theta_{f\otimes Z}^2=\theta_Z^2$ for $Z\in\cB$ by \cite{BRWZ}, restrict $\rho$ to $\Gamma_\theta$ and   to $W$,  then Equation \eqref{eq:center-representation} follows from Equation 
\eqref{eq:relative-data}.

Let
$G^{\cM}_\sigma=\sigma(s_{\cM})s_{\cM}^{-1}.
$
The center Galois identities \cite[Lemma~4.2]{DLN} give
\begin{equation}\label{equationcenterGalois}
 G^{\cM}_{\sigma_a}=\rho(D_a),
 \qquad
 \sigma_a^2(T_{\cM})
   =G^{\cM}_{\sigma_a}T_{\cM}(G^{\cM}_{\sigma_a})^{-1}.
\end{equation}
Here the first equality can be read in the congruence quotient:
if $ab\equiv1\pmod N$, the element
$\hxt^a\hxs\hxt^b\hxs \hxt^a\hxs^{-1}$
reduces to $\text{diag}(a,b)$ modulo $N$ and belongs to $\Gamma_\theta$. Because $W$ has a rational basis, restriction commutes with Galois
action.  Cyclotomic automorphisms commute with complex conjugation,
and $G_\sigma$ is real.  Hence
\[
 G^{\cM}_\sigma\big|_W
 =\sigma(s\otimes\overline s)(s\otimes\overline s)^{-1}
 =G_\sigma\otimes G_\sigma.
\]
Square the second identity in Equation \eqref{equationcenterGalois} and restrict it
to $W$, then the diagonal entry indexed by $(X,Y)$ is exactly $\sigma_a^2(\theta_X^2/\theta_Y^2)=\theta_{\hat{\sigma}_a(X)}^2/\theta_{\hat{\sigma}_a(Y)}^2$, the remaining equation follows from $Y=\one$.
\end{proof}

Let $\cC$ be a super-modular fusion category of global dimension $2p$. By the Cauchy theorem for spherical fusion categories
\cite[Theorem~3.9]{BNRW}, the prime divisors of $N_0:=\FSexp(\cC)$ are precisely
$2$ and $p$.  Thus we may write
$
 N=4N_0=2^A p^B, A\ge2, B\ge1.
$
All the matrices and automorphisms used above may be taken over
$\Q(\zeta_N)$: the entries of $\widehat S$ and the twists lie there,
and $\sqrt p\in\Q(\zeta_{4p})\subseteq\Q(\zeta_N)$.

\begin{lemma}\label{lemmaPpower}
    Let $\cC$ be a super-modular fusion category of global dimension $2p$, where $p>5$ is a prime. If $[\K:\Q]>1$, then at least one $\theta_X^2$ has order divisible by $p$ for some $X\in \I_0$.
\end{lemma}

\begin{proof}
Suppose instead that all $u_i$ have $2$-power order.
Then $\phi(\mathsf t^2)^{2^A}=1$.
By using  the  Chinese remainder theorem, it was proved in \cite[Lemma~2.2]{BRWZ} that $
 \Gamma_\theta/\Gamma(N)
 \simeq H_2\times\SL_2(\Z_{p^B})$,
where $H_2$ is the $2$-Sylow subgroup of $\SL_2(\Z_{2^A})$.
The element $\mathsf t^{2^{A+1}}$ has trivial $2$-component and
its $p$-component is $
\begin{pmatrix}1&2^{A+1}\\0&1\end{pmatrix}$, which then generalizes the entire
upper unipotent subgroup,
 as $2^{A+1}$ is a unit modulo $p^B$.  The normal closure of this subgroup in
$\SL_2(\Z_{p^B})$ is the whole group: conjugation gives the lower
unipotent subgroup, and the two elementary subgroups generate
$\SL_2(\Z)$.  Hence $\phi$ is trivial on the $p$-factor.

Choose $a$ such that
$
 a\equiv1\pmod{2^A}, 
 a\text{ is a primitive root modulo }p^B.
$
Then $D_a$ lies in the $p$-factor.  By Equation
\eqref{equationdiagonalGaloistwist},
$
 G_{\sigma_a}\otimes G_{\sigma_a}=1
$; notice that an invertible matrix whose tensor square is the identity matrix is $\pm 1$, so $\sigma_a(s)=\pm s$, in particular,  $
 \sigma_a(d_X)=
 \sigma_a\left(\frac{s_{\one,X}}{s_{\one,\one}}\right)=d_X$.
However, since 
 $a$ is primitive modulo $p$, $\sigma_a$ restricts to a generator of $\Gal(\K/\Q)$,
thus $\K=\Q$, it is impossible as $[\K:\Q]>1$ by assumption.
\end{proof}

\begin{theorem}\label{supermodularpointed}
    Let $\cC$ be a super-modular fusion category such that $N_{\Q(\dim(\cC)/\Q)}\left(\dim(\cC)/2\right)=p$, where $p>5$ is a prime. Then $\cC$ is pointed.
\end{theorem}
\begin{proof}
If $\dim(\cC)\notin\Z$, this is proved in Proposition \ref{propositionsquarefree}. 
If $\dim(\cC)=2p$ and suppose that  $\cC$ is not pointed.  Proposition~\ref{propositionfield} implies
$2\le h<\sqrt p$, $hm\le p$, and $\K\subseteq \Q(\zeta_p)^+$ with $h\mid n$, where $n:=\frac{p-1}{2}$.
The case $p=7$ has already been excluded by Corollary \ref{corollaryp=7}, so $p\ge11$.

Choose $a\equiv1\pmod{2^A}$ primitive modulo $p^B$, and write
$\sigma=\sigma_a, \hat{\sigma}=\hat{\sigma}_a,g=a\bmod p^B$.
Let $(\theta_X^2)_p$ be the $p$-primary component of $\theta_X^2$, $\forall X\in \cO(\cB)$. Let 
$(\theta_X^2)_p=\zeta_{p^B}^{b_X},
 b_X\in\Z/p^B\Z$ and $b_\one=0$. 
Equation~\eqref{equationdiagonalGaloistwist} gives the  relation $
 b_{\hat{\sigma}(X)}=g^2b_X+c$, where $c=b_{\hat{\sigma}(\one)}$.
Since $\sigma|_\K$ has order $h$, 
$\hat{\sigma}^h(\one)=\one$.  Iterating the above equation at the unit object $\one$ yields \begin{align*}
 0=c\bigl(1+g^2+\cdots+g^{2(h-1)}\bigr)
 \quad\text{in }\Z_{p^B}.
 \end{align*}
The inequalities $h<\sqrt p$  and $p>5$ imply $2h<p-1$.
Because $g$ is primitive modulo $p$, both $g^2-1$ and $g^{2h}-1$
are nonzero modulo $p$.  Therefore,
$
 1+g^2+\cdots+g^{2(h-1)}
 =\frac{g^{2h}-1}{g^2-1}
$
is a unit modulo $p^B$, so $c=0$.  Consequently,
$
 b_{\hat{\sigma}(X)}=g^2b_X$.

By Lemma~\ref{lemmaPpower}, there exists a simple object $X$ such that $(\theta_X^2)_p$ has
order $p^e$ with $e\ge1$, equivalently, $b_X=p^{B-e}t$ and $p\nmid t$. For any positive integer $r$, notice that $g^{2r}b_X\equiv b_X \pmod {p^B}$ if and only if $g^{2r}\equiv 1 \pmod {p^e}$, so the order of $g^2$  modulo $p^e$ is 
$ L=p^{e-1}n$.  Hence
$X,\hat{\sigma}(X),\ldots,\hat{\sigma}^{L-1}(X)$ are distinct objects.
None is $\one$, since all their $p$-primary twist components are
nontrivial.

By the Galois symmetry, we know 
$d_{\hat{\sigma}^j(X)}^2=\sigma^j(d_X^2)$.
Since $h\mid n\mid L$ and $\sigma|_\K$ generates $\Gal(\K/\Q)$,
\begin{align*}
 p-1
 \ge\sum_{j=0}^{L-1}d_{\hat{\sigma}^j(X)}^2
 =\frac{L}{h}\Tr_{\K/\Q}(d_X^2)
 =p^{e-1}\Tr_{\Q(\zeta_p)^+/\Q}(d_X^2).
\end{align*}
If $d_X\notin\Q$, Lemma~\ref{lemmatrace} makes the last expression
at least $p$, a contradiction.  Thus $d_X\in\Z$, which means
 $d_X^2=1$ by \cite[Corollary 3.4] {Yu2020}, so every object in the orbit of $X$ therefore has squared
dimension $1$.
Notice that the unit object $\one$ also has squared dimension $1$ and lies outside
the orbit, thus
$
 m\geq L+1\geq\frac{p+1}{2}.
$
However, we have 
$
 p\ge hm\ge p+1
$ by Proposition \ref{propositionfield} , impossible.  Hence $\cC$ is pointed.
\end{proof}

\section{Modular fusion category of global dimension  $2p$}\label{sectionmodular}
 In this section, we classify modular fusion categories of global dimension $2p$, where $p>5$ is a prime.

\begin{lemma}\label{lem:perfect}
Let $\cC$ be a non-pointed modular fusion category of global dimension $2p$ with prime $p>5$. Then
$\cC_{\pt}=\Vec$ and $\cC_{\ad}=\cC$.
\end{lemma}
\begin{proof}
Let $\mathcal{E}$ be a non-trivial fusion subcategory of $\cC$. Then the M\"{u}ger center $\mathcal{E}'$ of $\mathcal{E}$ is a symmetric fusion subcategory of $\cC$,  which must be trivial as $\dim(\mathcal{E}')^2\mid \dim(\cC)$ by \cite[Theorem 3.10]{DrGNO1}, that is, $\mathcal{E}$ is modular.  In particular, the pointed subcategory $\cC_{\pt}$ is therefore modular. Then $\cC\cong\cC_\pt\boxtimes\cC_\ad$ as $(\cC_\pt )'_\cC=\cC_\ad$ \cite[Corollary 3.27]{DrGNO1}. Assume $\cC_\pt$ is non-trivial below. 
 Notice that $\cC$ is not pointed, so $\dim(\cC_\pt)=2,p$. Then  $\dim(\cC_\ad)$ also is a prime, which then implies $\cC_\ad$ is pointed by \cite[Theorem 5.12]{Schopieray2021}, it is contradiction.
\end{proof}

Let $T$ be  the
unnormalized twist matrix of $\cC$ with   $N=\operatorname{ord}(T)$.
 For any linear modular lift
$(s,t)$ with $n=\operatorname{ord}(t)$,   the
representation factors through $\SL_2(\Z_n)$, 
$N\mid n\mid12N$  and 
\begin{align}\label{Galois}
 G_a=\sigma_a(s)s^{-1}
=\rho\begin{pmatrix}a&0\\0&a^{-1}\end{pmatrix},\quad t_{\widehat\sigma_a(X)}=t_X^{a^2},\quad \forall X\in\cO(\cC), 
\end{align}
where $\sigma_a(\zeta_n)=\zeta_n^a$ and $G_a$ is a signed permutation
matrix \cite[Theorem II]{DLN}. Also $S$ has entries in $\Q(\zeta_N)$. The diagonal matrix is interpreted modulo
$n$. Moreover,  $G_a=t^a s t^b s t^a s^{-1}$, with
$ab\equiv1\pmod n$, is exactly this diagonal element.

\begin{proposition}\label{propositionlevel}
Let $\cC$ be a modular fusion category of global dimension $2p$ with $p>5$. Then there is a linear modular lift $\rho_\cC$ with
$
 \operatorname{ord}(T)\mid4p,\
 \operatorname{ord}(t)\mid8p,
$ where $t=\rho_\cC(\hxt)$.
Therefore, we have 
$L\subseteq\Q(\zeta_{4p})^+$.
\end{proposition}
\begin{proof}
The   Vafa's theorem implies $ 
 N\mid D^{5/2}$ by \cite[Corollary 8.18.2]{EGNO}, so
$N\mid4p^2$. If $p^2\mid N$, let $t=\rho_\cC(\mathfrak{t})$, where $\rho_\cC$ is a lift of the projective congruence representation of $\cC$, then there exists a simple object $X$ of $\cC$ such that  $p^2\mid \ord (t_{X,X})$, so Equation~\eqref{Galois}
would imply
$ \rank(\cC)\geq\frac{p(p-1)}2>2p$, it is impossible. Thus $N\mid4p$, thus $d_X\in\Q(\zeta_{4p})^+$.

Notice that the Gauss sum $\tau^+(\cC)$ lies in $\Q(\zeta_{4p})$,
so the multiplicative central charge
$\xi(\cC)=\tau^+/\sqrt{2p}$, which is a root of
unity, lies in $\Q(\zeta_{8p})$. Meanwhile, the map $\xi\mapsto\xi^3$  
defines an automorphism   the group $\mu_{8p}$ of roots of unity in  $\Q(\zeta_{8p})$ as $p>5$. Choose
$\gamma\in\mu_{8p}$ with $\gamma^3=\xi$. Then
\[
 s=S/\sqrt{2p},\qquad t=\gamma^{-1}T
\]
defines a linear modular representation $\rho_\cC$ with $\ord(t)\mid8p$.
\end{proof}
The Chinese Remainder Theorem   shows the congruence representation $\rho_\cC$ factors through $\SL_2(\Z_{8p})\cong\SL_2(\Z_8)\times\SL_2(\Z_p)$ by Proposition \ref{propositionlevel}. 
\begin{proposition}\label{prop:even-local}
 Let $\chi\otimes \phi$  be an irreducible summand   of $\rho_\cC$. Then $\dim(\chi)$ is even and 
$\dim (\phi)\leq p$. In particular, $\rank(\cC)$ is even.
\end{proposition}
\begin{proof}
Let $a\in\Z$ such that  $a\equiv5\pmod8$ and $a\equiv1\pmod p$.
Then the  automorphism $\sigma_a$ fixes $\Q(\zeta_{4p})$
as  $a\equiv1\pmod{4p}$ and $\sigma_a(\sqrt{2})=    -\sqrt{2}$. Therefore $\sigma_a(\sqrt{2p})=-\sqrt{2p}$, which implies
\begin{equation}\label{eq:minus}
G_a=\sigma_a(s)s^{-1}=-I.
\end{equation}
The  corresponding  diagonal element of $\SL_2(\Z_8)\times\SL_2(\Z_p)$ in Equation
\eqref{Galois} is $(5I_2,I_2)$.
In the subgroup $\SL_2(\Z/8\Z)$, notice that  we have 
\begin{equation}\label{eq:commutator}
[u,v]:=uvu^{-1}v^{-1} =5I_2, \quad \text{where} ~u=
 \begin{pmatrix}1&2\\0&1\end{pmatrix},\quad v=
 \begin{pmatrix}1&0\\2&1\end{pmatrix},
\end{equation}
 Then for the irreducible summand  $\chi\otimes \phi$,  Equation~\eqref{eq:minus}
forces $\chi(5I_2)=-I_A$. Taking determinants in
Equation \eqref{eq:commutator} gives $(-1)^{\dim A}=1$, hence $\dim (\chi)\geq2$ is even. Since
$\dim (\chi)\dim (\phi)\leq \rank(\cC)\leq2p$,  $\dim (\phi)\leq p$.
\end{proof}

\begin{proposition}\label{propositionquadratic}
Let $\cC$ a modular fusion category of global dimension $2p$ with prime $p>5$. If $h=[\K:\Q]\leq2$, then $\cC$ is pointed.
\end{proposition}
\begin{proof}
If $\K=\Q$,   $\cC$ is pointed by Lemma \ref{lemmadimension}. 
Assume $[\K:\Q]=2$ below,  notice that  the unique quadratic subfield of
$\Q(\zeta_{4p})^+$ is $\Q(\sqrt p)$, so $\K=\Q(\sqrt{p})$. By
Lemma~\ref{lem:perfect}, we have
$\cC_{\ad}=\cC$. 

Assume $p\equiv3\pmod4$, write $d_X=a_X+b_X\sqrt p$ with integers
$a_X,b_X$. 
 Let $\sigma\in\Gal(L/\Q)$ such that $\sigma(\sqrt{p})=-\sqrt{p}$, then $\sum_{X\in\cO(\cC)}d_X\sigma(d_X)=0$ by \cite[Lemma 8.14.1]{EGNO}, so  we have 
\[
 \sum_{X\in\cO(\cC)} a_Xb_X=0,\qquad
 \sum_{X\in\cO(\cC)} a_X^2=p,\qquad \sum_{X\in\cO(\cC)} b_X^2=1.
\]
Therefore, there is exactly one simple object $X$ such that $b_X\neq0$ and $a_X=0$, and  $d_Y=\pm1$ for all other simple objects $Y\in\cO(\cC)$, which then implies  $\FPdim(\cC)=\frac{\dim(\cC)}{d_X^2}=2p$ as it is the largest formal codegree of $\cC$, so $\cC$ is integral and pointed by  \cite[Theorem 2.8, Lemma 3.5]{Yu2020}, it is a contradiction.

Suppose $p\equiv1\pmod4$. Let
$d_X=(a_X+b_X\sqrt p)/2$, where $a_X,b_X\in\Z$.
The same argument gives
\[
\sum_{X\in\cO(\cC)}a_Xb_X=0,\qquad
 \sum_{X\in\cO(\cC)}a_X^2=4p,\qquad \sum_{X\in\cO(\cC)}b_X^2=4.
\]
If $b_X^2=4$ for some simple object $X$, then we again obtain $\FPdim(\cC)=2p$ as before. Hence,  there are exactly four simple objects with nonzero $b_X$, each equal to $\pm1$. Then $N_{\K/\Q}(d_X)=\frac{a_X^2-p}{4}$, so $a_X$ is odd and $N_{\K/\Q}(d)\mid 2p$.  Let $f=\frac{\dim(\cC)}{d_X^2}$ be the corresponding formal codegree of $\cC$, so $\Tr_{\K/ \Q}(f)=\frac{p(a_X^2+p)}{N_{\K/\Q}(d)^2}$. Since $a_X$ is odd and $p\equiv 1\pmod 4$, $p(a^2+p)\equiv 2\pmod 4$. If $N_{\K/\Q}(d)$ is even, then $4\mid N_{\K/\Q}(d)^2$, which contradicts to $p(a^2+p)\equiv 2\pmod 4$. Therefore, 
we have \(N_{\K/\Q}(d)\in\{\pm1,\pm p\}\).
And then compare this with \(N_{\K/\Q}(d)=\frac{a^2-p}{4}\).
If  $N_{\K/\Q}(d)=\pm p$, then $a^2=5p$ or $a^2=-3p$, which is impossible for
$p>5$. If $N_{\K/\Q}(d)=1$, then $p=a^2-4=(|a|-2)(|a|+2)$, 
impossible. Thus $N_{\K/\Q}(d)=-1$ and $a^2=p-4$ for all
four simple objects.

Then leaves $16$ for the sum of
$a_X^2$ of simple  objects whose dimensions are rational, so
$a_X^2=4$ for these objects, which then means that there are exactly four such objects. The squared dimensions of the remaining four simple objects  $d^2$ and its conjugates with  multiplicities being  two, respectively.  However, 
 the largest formal codegree $\FPdim(\cC)$ then  occurs twice,
contradicting Lemma~\ref{lemmaFPdimension}.
\end{proof}

\begin{lemma}\label{lemmaunitarymatrix}
Let $A$ and $\varphi$ be unitary operators on a finite-dimensional complex
Hilbert space, and let $B:=\varphi A\varphi^{-1}$. Suppose $AB=-BA$.
If $w$ is a unit vector and $Aw=\lambda w$, then $
 |\langle w,\varphi (w)\rangle|^2\leq\frac{1}{2}$.
\end{lemma}
\begin{proof}
The vectors $w$ and $Bw$ are eigenvectors of
$A$ with distinct eigenvalues $\lambda$ and $-\lambda$, respectively, so they are orthogonal. Denote $x:=\varphi(w)$, then $B(x)=\lambda x$ and $
 |\langle B(w),x\rangle|=|\langle w,x\rangle|$.
The Bessel's inequality for the orthonormal pair $w,B(w)$ gives $
 2|\langle w,\varphi (w)\rangle|^2
 =|\langle w,x\rangle|^2+|\langle B(w),x\rangle|^2
 \leq\|x\|^2=1$.
\end{proof}

\begin{proposition}\label{propositionfieldbound}
Let $\cC$ be a  modular fusion category of global
dimension $2p$, where $p>5$. Then $h^2\leq p$.
\end{proposition}
\begin{proof}
Let $(s,t)$ be  the unitary modular lift of level dividing $8p$ from
Proposition~\ref{propositionlevel}. Notice that  the  restriction of $\Gamma:=\Gal(\Q(\zeta_{8p})/\Q(\zeta_8))\cong\Z_p^\times$ to $\Q(\zeta_{4p})^+$ is an isomorphism, then 
Thus $\Gal(\Q(\zeta_{8p})/\Q(\zeta_8))$ restricts surjectively to $\Gal(\K/\Q)$.
Let $\mathcal O_\one$ be the $\Gamma$-orbit of the unit object $\one$, then $\cO_\one$ has $[\K:\Q]$ simple objects and $d_X^2=1$ for  all $X\in\cO_\one$.
Moreover, for $X,Y\in\mathcal O_\one$, $
 S_{X,Y}=d_Xd_Y$.
Indeed, let $h_Y(-):=\sigma\circ\dim(-)$. Then
$S_{X,Y}/d_Y=\sigma(d_X)=d_X$ as $d_X=\pm1$.

Define $
 w=\frac1{\sqrt h}\sum_{X\in\mathcal O_\one}d_Xe_X$, where $h=[\K:\Q]$. Then 
 we have 
\begin{equation}\label{eq:vacuum-matrix-coefficient}
 \|w\|=1,\qquad \langle w,sw\rangle=\frac{h}{\sqrt{2p}}.
\end{equation}
Choose an integer $m$ such that $m\equiv2\pmod8$ and
$m\equiv0\pmod p$, and set
\[
 A:=t^m=\rho\left(\begin{pmatrix}1&2\\0&1\end{pmatrix},I_2\right),
 \qquad B=sAs^{-1}.
\]
All labels in $\mathcal O_\one$ have the same $A$-eigenvalue.
In fact, for $X=\widehat\sigma_a(\mathbf1)$ with
$a\equiv1\pmod8$, Galois symmetry gives
$t_X=t_{\mathbf1}^{a^2}$. Because $p\mid m$ and $8\mid a^2-1$,
we have $8p\mid m(a^2-1)$, and therefore
$t_X^m=t_{\mathbf1}^m$.
Consequently $Aw=\lambda w$ with $\lambda=t_{\mathbf1}^m$.

Let $
u=\begin{pmatrix}1&2\\0&1\end{pmatrix},
 v=\begin{pmatrix}1&0\\-2&1\end{pmatrix}$. 
Conjugation by the standard modular generator $\hxs$ sends $u$ to $v$,
and direct multiplication gives $[u,v]=5I_2$ modulo $8$.
Thus Equation~\eqref{eq:minus} implies
\[
 ABA^{-1}B^{-1}=\rho_\cC(5I_2,I_2)=-I,
 \qquad AB=-BA.
\]
Lemma~\ref{lemmaunitarymatrix} and Equation
\eqref{eq:vacuum-matrix-coefficient} now yield $\frac{h^2}{2p}\leq\frac{1}{2}$,
which is precisely $h^2\leq p$.
\end{proof}


\begin{lemma}\label{lemmaborel}
Let $W$ be an irreducible complex representation of $G=\SL_2(\Z_p)$
with $\dim (W)\leq p$. Set
\[
 U=\left\{\begin{pmatrix}1&x\\0&1\end{pmatrix}:x\in\Z_p\right\},
 \qquad
 H=\left\{\begin{pmatrix}b^2&0\\0&b^{-2}\end{pmatrix}:b\in\Z_p^\times\right\}.
\]
Then $H$ acts trivially on $W^U$.
\end{lemma}
\begin{proof}
Let $A$ be the full diagonal torus and $ B:=AU$ the upper
triangular Borel subgroup. Then the $A$-module $W^U$ is a sum of character
spaces. If a character $\lambda$ occurs, extend it to $B$ by
making $U$ act trivially. Frobenius reciprocity shows that $W$ is a summand of
$\Ind_{B}^G(\lambda)$. Let $s:=\begin{pmatrix}0&-1\\1&0\end{pmatrix}$, then there are two Bruhat double cosets, namely $B$ and $BsB$. Notice that for any element $u\in A$, we have $s^{-1}us=u^{-1}$, so nontrivial Weyl element acts on
$A$ by inversion \cite{Serre}. Then the Mackey's formula gives
\[\dim\operatorname{End}_G(\Ind_{B}^G\lambda)
 =1+\delta_{\lambda,\lambda^{-1}}.
\]
If $\lambda^2\ne1$, the induced representation is irreducible, of
dimension $[G:B]=p+1$. This contradicts $\dim( W)\leq p$.
Thus every character occurring in $W^U$ has square one and is trivial
on $H$.
\end{proof}

\begin{theorem}\label{modularpointed}
Let $\cC$ be a modular fusion category of global dimension $2p$, where   $p>5$ is a prime. Then $\cC$ is pointed.
\end{theorem}
\begin{proof}
Let $\rho_\cC$ be the lift of level dividing $8p$ in Proposition \ref{propositionlevel}. Let $e_{\mathbf1}$ be
 the basis vector of the  unit object, then its  $t$-eigenvalue is
$t_{\mathbf1}=\gamma^{-1}$.
First, this eigenvalue has trivial $p$-part. Otherwise,  by
Equation \eqref{Galois}, the automorphisms
$\sigma_a$ with $a\equiv1\pmod8$ would produce $(p-1)/2$ different eigenvalues in the
Galois orbit $\cO_\one$, which  has cardinality $[\K:\Q]$. However,  it follows from Proposition~\ref{propositionfieldbound} that
\[
 \frac{p-1}{2}\leq [\K:\Q]\leq\sqrt p,
\]
 it is impossible. Thus
$t_{\mathbf1}^8=1$. Notice that the eighth power of the standard upper
unipotent element 
in $\SL_2(\Z_8)\times \SL_2(\Z_p)$ is
$\left(I_2,\begin{pmatrix}1&8\\0&1\end{pmatrix}\right)$. Since $8$ is nonzero modulo $p$, its second component generates $U$.
It follows that $e_{\mathbf1}\in V_\cC^{\{1\}\times U}$.
Decompose $V_\cC$ into irreducible constituents $\chi\otimes \phi$.
By Proposition~\ref{prop:even-local}, each $\phi$ has dimension at most
$p$. Lemma~\ref{lemmaborel} says that $\{1\}\times H$ fixes the whole
space $V_\cC^{\{1\}\times U}$, and hence fixes $e_{\mathbf1}$.

Now take $a\equiv1\pmod8$ whose residue modulo $p$ is a square.
Equation~\eqref{Galois} gives $G_a e_{\mathbf1}=e_{\mathbf1}$.
The Galois permutation therefore fixes the $e_\one$, so $\sigma_a(d_X)=d_X$ for every $X$, then the restrictions of these automorphisms to
$\Q(\zeta_{4p})^+$ form a subgroup of order $(p-1)/2$. Meanwhile,  $\K\subseteq 
\Q(\zeta_{4p})^+$ by Proposition \ref{propositionlevel} and this
subgroup fixes $\K$ pointwise, so $[L:\Q]\leq2$ , and $\cC$ is pointed by Proposition \ref{propositionquadratic} .
\end{proof}

Combing the conclusions  of Theorem \ref{supermodularpointed} and Theorem \ref{modularpointed} with that of \cite{Ostrik,Yu2022}, we have the following 
\begin{corollary}
Let $\cC$ be a pre-modular fusion category of global dimension $2p$, where $p$ is a prime. Assume $\cC'\subseteq\sVec$, then $\cC$ is pointed if $p\neq2,5$. Thus, either  $\cC$ is pointed, or
\begin{align*}
\cC\cong\left\{
  \begin{array}{ll}
    \mathcal{I}, & \hbox{ if  $p=2$;} \\
    \cB\boxtimes \cC(\mathfrak{sl}_2,3)_\ad\boxtimes\cC(\mathfrak{sl}_2,3)_\ad^\sigma, & \hbox{ if $p=5$,}
  \end{array}
\right.
\end{align*}
where $\cB$ is a pointed pre-modular fusion category of dimension $2$ and $\cC(\mathfrak{sl}_2,3)_\ad^\sigma$ is a Galois conjugate of $\cC(\mathfrak{sl}_2,3)_\ad^\sigma$ with $\dim(\cC(\mathfrak{sl}_2,3)_\ad^\sigma)=\frac{5-\sqrt{5}}{2}$.
\end{corollary}
 \section*{Acknowledgments}
 Z. Yu  is supported
by NSFC (no.12571041) and Qinglan Project of Yangzhou University. The author acknowledges the use of AI in the development
of this manuscript, including mathematical discussion,  suggestions for  proof steps, and corrections. All
mathematical content was independently verified by the author, who takes full responsibility for the final manuscript.


\noindent \author{Zhiqiang Yu\\ \thanks{Email:\,zhiqyumath@yzu.edu.cn}
\\{\small School  of Mathematics,  Yangzhou University,
Yangzhou 225002, China}
}
\end{document}